\documentclass[letterpaper,10pt]{article}

\usepackage[square,numbers]{natbib}

\usepackage[utf8]{inputenc}   
\usepackage[T1]{fontenc}  

\usepackage{pgf,tikz}
\usetikzlibrary{arrows}
\usepackage{multicol}
\usepackage{amsthm}
\usepackage{amsmath}
\usepackage{amssymb}
\usepackage{amsfonts}
\usepackage{stmaryrd}
\usepackage{mathabx}
\usepackage{mathtools}
\usepackage{latexsym}
\usepackage{color}
\usepackage{graphics,graphicx,graphpap}
\usepackage{multirow}
\usepackage{rotating}
\usepackage[new]{old-arrows}
\usepackage[all]{xy}
\usepackage[letterpaper]{geometry}
\usepackage{subfigure}
\usepackage{hyperref}
\usepackage{orcidlink}
\usepackage{authblk}
\usepackage{diagbox}
\theoremstyle{plain}

\usepackage[nottoc]{tocbibind}

\DeclareMathAlphabet{\mathpzc}{OT1}{pzc}{m}{it}
\newtheorem{theorem}{Theorem}
\newtheorem{prop}[theorem]{Proposition}
\newtheorem{cor}[theorem]{Corollary}
\newtheorem{lem}[theorem]{Lemma}

\newcommand*{\email}[1]{%
    \normalsize\href{mailto:#1}{#1}\par
    }
\title{Independence complexes of generalized Mycielskian graphs}
\author{Andr\'es Carnero Bravo\;\orcidlink{0009-0001-0221-5908}}
\affil{Centro de Ciencias Matemáticas, UNAM\\ \email{carnero@matmor.unam.mx}}

\begin{document}
\maketitle
\begin{abstract}
We show that the homotopy type of the independence complex of the generalized Mycielskian of a graph $G$ is determined by the homotopy 
type of the independence complex of $G$ and the homotopy type of the independence complex of the Kronecker double cover of $G$. 
As an application we calculate the homotopy 
type for paths, cycles and the categorical product of two complete graphs.
\end{abstract}
\textbf{Acknowledgments.} This work was supported by UNAM Posdoctoral Program (POSDOC) and by the project PAPIIT IN101423.
\section{Introduction}
In this note we will study the independence complex of the Mycielskian and the generalized Mycielskian of a graph $G$.
The Mycielski graphs were defined to construct examples of graphs without triangles and chromatic number arbitrarily big 
\citep{mycielski}. This construction has been widely studied (see for example \citep{alishahi,dliou,dovslic,fisherdavid,huang,larsen}). 
Later, the construction was generalized 
--originaly called \textit{cones over graphs} \citep{tardif}-- and this generalization has been considerably studied 
(see for example \citep{bidine,lin2010,lin2006}).  

The independence complex of a graph is an extensively studied simplicial complex where the simplices are the independent sets 
(see for example \citep{barmak,Ehrenborg_2006,engstrom09,meshulandom}). Much of the work have been focused in particular graph families 
(see for example \citep{indcomplcatprod,indpcomplgrpprod,homotopygoyal,matsushitan4n5,matsushitan6}).

Our main result is that the independence complex of the generalized Mycielskian of $G$
has the homotopy type of a wedge of spaces where 
each space will be an iterated suspension and/or join of two simplicial complexes: $I(G)$ and $I(G\times P_2)$. 
The graph $G\times P_2$ is the 
categorical product of the graph $G$ and the path of length one, this graph is called the 
\textit{Kronecker double cover} or the \textit{canonical double cover}, this construction also has been significantly studied 
(see for example \citep{gevaykrncv,imrichmultkrncov,krncgenptskrcv,matsushitaegbcomplkrdcovers,matsushitagkrnbipknser,wallerdobcover}).
We also give a formula for the homotopy type of the independence complex of the Kronecker double cover generalized Mycielskian, with this formula we are able to give 
a formula for the iterated generalized Mycielskians $\mu_{3l+1}^r(G)$ and $\mu_{3l+2}^r(G)$.

As far as we know the only results on the independence complex of 
Mycielscki graphs are in \citep{homotopygoyal}, there its shown that the independence complex of the usual Mycielskian of $G$ has the 
homotopy type of the suspension of the independence complex of $G$ and that the independence complex of the generalized 
Mycielsckian of a complete graph has the homotopy type of a wedge of spheres. We obtain these results as corollaries of 
our main result 
Theorem \ref{teomul}. As example, we determine the homotopy type of the independence complex of generalized Mycielsckians for paths, 
cycles and the categorical product of two complete graphs.

\section{Preliminaries}
We are only interested in simple graphs. Given a graph $G$, $V(G)$ is its vertex set and $E(G)$ is its edge set.
Taking the set $\underline{n}=\{1,\dots,n\}$, we define three graphs with vertex set $\underline{n}$:
\begin{itemize}
    \item The \textit{path} on $n$ vertices is the graph $P_n$ with edge set $E(P_n)=\{\{i,j\}:\;|i-j|=1\}$.
    \item The \textit{cycle} on $n\geq3$ vertices is the graph $C_n$ with edge set $E(C_n)=\{\{i,j\}:\;|i-j|=1\}\cup\{1n\}$.
    \item The \textit{complete graph} $K_n$ is the graph with edge set $E(K_n)=\{\{i,j\}:\; i\neq j\}$.
\end{itemize}
Given two graphs $G$ and $H$, their 
\textit{categorical product} is the graph $G\times H$ with vertex set $V(G)\times V(H)$ and where $\{(u,v),(x,y)\}$ is 
an edge if $\{u,x\}$ is an edge of $G$ and $\{v,y\}$ is an edge of $H$. Notice that if $G$ is bipartite then $G\times P_2\cong G\sqcup G$, 
\textit{i.e.} $G\times P_2$ is isomorphic to the disjoint union of two copies of $G$. For any graph $G$, $G\times P_2$ is 
called the \textit{Kronecker double cover} of $G$.

Given a graph $G$ and $l\geq0$, we define the $l$-\textit{Mycielskian} of $G$ as the graph $\mu_l(G)$ (see Figure \ref{mulg})
obtained from $G\times P_{l+1}$ by:
\begin{itemize}
    \item Adding a new vertex $w$ and all the possible edges between $w$ and $V(G)\times\{1\}$.
    \item Adding the edges $\{(u,l+1),(v,l+1)\}$ for all $\{u,v\}\in E(G)$.
\end{itemize}
Notice that $\mu_0(G)$ is the \textit{cone} of $G$.  
For $r=1$ we take $\mu_l^1(G)=\mu_l(G)$ and for $r\geq2$ we take $\mu_l^r(G)=\mu_l\left(\mu_l^{r-1}(G)\right)$. We call $\mu_l^r(G)$ the $r$-\textit{iterated} $l$-\textit{Mycielskian} of $G$.

\begin{figure}
\centering
\begin{tikzpicture}[line cap=round,line join=round,>=triangle 45,x=1.0cm,y=1.0cm]
\clip(2,3) rectangle (10,5);
\draw [rotate around={90:(3,4)}] (3,4) ellipse (0.71cm and 0.5cm);
\draw [rotate around={90:(8.5,4)}] (8.5,4) ellipse (0.71cm and 0.5cm);
\draw [rotate around={90:(4.5,4)}] (4.5,4) ellipse (0.71cm and 0.5cm);
\draw [rotate around={90:(7,4)}] (7,4) ellipse (0.71cm and 0.5cm);
\draw (3.35,3.5)-- (4.15,4.5);
\draw (4.15,3.5)-- (3.35,4.5);
\draw (8.15,4.5)-- (7.35,3.5);
\draw (8.15,3.5)-- (7.35,4.5);
\draw (9.5,4)-- (8.85,4.5);
\draw (9.5,4)-- (9,4);
\draw (9.5,4)-- (8.85,3.5);
\begin{scriptsize}
\fill [color=black] (9.5,4) circle (1.5pt);
\draw [color=black] (9.5,4.2) node {$w$};
\draw [color=black] (3,4) node {$G$};
\draw [color=black] (4.5,4) node {$V_l$};
\draw [color=black] (7,4) node {$V_{2}$};
\draw [color=black] (8.5,4) node {$V_{1}$};
\draw [color=black] (5.75,4) node {$\cdots$};
\draw [color=black] (5.5,4) node {$\cdots$};
\draw [color=black] (5.25,4) node {$\cdots$};
\draw [color=black] (6,4) node {$\cdots$};
\draw [color=black] (6.25,4) node {$\cdots$};
\end{scriptsize}
\end{tikzpicture}
\caption{$\mu_l(G)$}
\label{mulg}
\end{figure}

The \textit{independence complex} of  graph $G$ is the simplicial complex $I(G)$ where a set of vertices $\sigma$ is a simplex 
if no two vertices in it are adjacent, \textit{i.e.} there is no edge between any two vertices in the set. We will not distinguish 
between the simplcial complex and its geometric realization. The join of $k$ copies of $I(G)$ will be denoted $I(G)^{*k}$, we will take $I(G)^{*0}$ as the simplicial complex $\{\emptyset\}$. 
$\Sigma^rI(G)$ means we are taking the complex $I(G)$ 
suspended $r$ times.
We will need the following results about the homotopy type of independence complexes. 

\begin{lem}\label{lemundisj}\citep{engstrom09}
If $G=G_1+G_2$, then $I(G)=I(G_1)*I(G_2).$
\end{lem}

\begin{lem}\citep{engstrom09}\label{lemmvertdad}
Let $u,v$ be different vertices of a graph $G$. If $N_G(u)\subseteq N_G(v)$, then $I(G)\simeq I(G-v)$.
\end{lem}
The last lemma is a particular case of the following proposition.
\begin{prop}\citep{adamsplit}\label{inclunullhmt}
If $I(G-N_G[v])\longhookrightarrow I(G-v)$ is null-homotopic, then $I(G)\simeq I(G-v)\vee\Sigma I(G-N_G[v])$.
\end{prop}

In Section 4 we will need some functions to count the number of suspensions in the homotopy type of iterated Mycielskians. 
For $k\geq0$ and $r\geq1$, we define:
$$f(k,r)=\left\lbrace\begin{array}{cc}
  r   &  \mbox{ if } k=0\\
  \frac{(2k+1)^r-1}{2k}   &  \mbox{ if } k\neq0
\end{array}\right.\;\;\mbox{ and}\;\;g(k,r)=\frac{(2k+2)^r-1}{2k+1}.$$

\begin{lem}\label{lemfg}
For $k\geq0$ and $r\geq1$ we have the following identities:
$$f(k,r)=\sum_{i=0}^{r-1}(2k+1)^i,\;\;\;g(k,r)=\sum_{i=0}^{r-1}(2k+2)^i$$
$$k\sum_{i=1}^{r-1}f(k,i)=\frac{f(k,r)-r}{2},\;\;\;(k+1)g(k,r)=\frac{g(k,r+1)-1}{2}$$
\end{lem}
\begin{proof}
For $k=0$ all the identities are clear. Assume that $k\geq1$. For the first two identities we will use induction over $r$. 
For $r=1$ is clear, we assume it is true for $r$. Then, for $r+1$ we have that
$$\sum_{i=0}^{r}(2k+1)^i=f(k,r)+(2k+1)^r=f(k,r+1)$$
$$\sum_{i=0}^{r}(2k+2)^i=g(k,r)+(2k+2)^r=g(k,r+1)$$
Now
$$k\sum_{i=1}^{r-1}f(k,i)=\frac{1}{2}\sum_{i=1}^{r-1}((2k+1)^{i-1}-1)=\frac{1}{2}\left(f(k,r)-r\right)$$
$$(k+1)g(k,r)=\frac{1}{2}\sum_{i=1}^{r}(2k+2)^i=\frac{1}{2}\left(g(k,r+1)-1\right)$$
\end{proof}
\section{Independence complex of \texorpdfstring{$l$}{Î£}-Mycielskians}
In this section we prove our main theorem and use it to calculate the homotopy type of the independence complex of the  $l$-Mycielskian
for various families of graphs.

\begin{theorem}\label{teomul}
For any graph $G$, 
$$I(\mu_l(G))\simeq\left\lbrace
\begin{array}{cc}
I(G)*I(G\times P_2)^{*k}\vee\Sigma I(G\times P_2)^{*k} & \mbox{if } l=3k\\
\Sigma I(G)*I(G\times P_2)^{*k} & \mbox{if } l=3k+1\\
I(G\times P_2)^{*k+1} & \mbox{if } l=3k+2
\end{array}\right.$$
\end{theorem}
\begin{proof}
For $l=0$, it is clear that $I(\mu_0(G))\cong I(G)\vee\mathbb{S}^0$. We assume that $l\geq1$.
We take $V_i=V(G)\times\{i\}$.
In $\mu_l(G)-x$, we have that $N_{\mu_l(G)}((v,1))\subseteq N_{\mu_l(G)}((v,3))$ for any $v$ in $V(G)$. Thus, by Lemma \ref{lemmvertdad}, 
$$I(\mu_l(G)-x)\simeq I(\mu_l(G)-x-V_3).$$
In $H=\mu_l(G)-x-V_3$, we have that $N_H((v,4))\subseteq N_H((v,6))$ for all $v$ in $V(G)$. Therefore, by Lemma \ref{lemmvertdad}, 
we can delete $V_6$. We keep doing this until we have that
$$I(\mu_l(G)-x)\simeq I(\mu_l(G)-x-V_3-V_6-\cdots-V_{3k}).$$
We take $W=\mu_l(G)-x-V_3-V_6-\cdots-V_{3k}$. Now we have the following $3$ cases: 
\begin{itemize}
    \item If $l=3k$, then $I(\mu_l(G)-x)\simeq I(W)\cong I(G)*I(G\times P_2)^{*k}$.
    \item If $l=3k+1$, then $I(\mu_l(G)-x)\simeq I(W)\cong I(G\times P_2)^{*k}*I(W')$ where 
    $$W'=\mu_l(G)[V(G)\times\{3k+1,3k+2\}].$$
    Now, we have that $N_{W'}((v,3k+1))\subseteq N_{W'}((v,3k+2))$ for all $v$ in $V(G)$, by Lemma \ref{lemmvertdad} we have that
    $I(W')\simeq I(W'-V_{3k+2})$.  Notice, that $W'-V_{3k+2}$ is the disjoint union of vertices, then $I(W'-V_{3k+2})$ is contractible. Therefore $I(\mu_l(G)-x)\simeq*$.
    \item If $l=3k+2$, then $I(\mu_l(G)-x)\simeq I(W)\cong I(G\times P_2)^{*k}*I(W')$ where $$W'=\mu_l(G)[V(G)\times\{3k+1,3k+2,3(k+1)\}].$$ 
    Now, we have that $N_{W'}((v,3k+1))\subseteq N_{W'}((v,3k+3))$ for all $v$ in $V(G)$,  by Lemma \ref{lemmvertdad} we have that
    $I(W')\simeq I(W'-V_{3k+3})\cong I(G\times P_2)$. Therefore $I(G-x)\simeq I(G\times P_2)^{*k+1}$.
\end{itemize}
With similar arguments it is easy to see that: 
$$I(\mu_{3k}(G)-N_{\mu_l(G)}[x])\simeq I(G\times P_2)^{*k},\;\;I(\mu_{3k+1}(G)-N_{\mu_l(G)}[x])\simeq I(G)*I(G\times P_2)^{*k}$$ and 
$I(\mu_{3k+2}(G)-N_{\mu_l(G)}[x])\simeq *$. For $l=3k$, we have that 
$$\mathrm{conn}\left(I(\mu_l(G)-x)\right)>\mathrm{conn}\left(I(\mu_l(G)-N_{\mu_l(G)}[x])\right),$$
thus the inclusion is null-homotopic. By Proposition \ref{inclunullhmt} we get the result.
\end{proof}

As first corollary we get the homotopy type of $I(\mu_0^r(G))$ and $I(\mu_1^r(G))$.
\begin{cor}
For any graph $G$ and $r\geq1$ we have that
$$I(\mu_0^r(G))\simeq I(G)\vee\bigvee_{r}\mathbb{S}^0\;\;\mbox{and}\;\;I(\mu_1^r(G))\simeq\Sigma^rI(G).$$
\end{cor}
For example, for $r\geq1$ we have that $I(\mu_1^r(K_2))\simeq\mathbb{S}^r$. It is worth remarking that the graphs $\mu_1^r(K_2)$ 
where constructed as an example of a triangle free graphs with chromatic number $r+2$ \citep{mycielski}.

Now we proceed to calculate the homotopy type for specific families of graphs. We obtain the following result from \citep[see Theorem 4.14.]{homotopygoyal} as a corollary of Theorem \ref{teomul}.
\begin{cor}\citep{homotopygoyal}
$$I(\mu_l(K_n))\simeq\left\lbrace
\begin{array}{cc}
\displaystyle\bigvee_{n(n-1)^k}\mathbb{S}^{2k} & \mbox{if } l=3k\\
\displaystyle\bigvee_{(n-1)^{k+1}}\mathbb{S}^{2k+1} & \mbox{if } l=3k+1\\
\displaystyle\bigvee_{(n-1)^{k+1}}\mathbb{S}^{2k+1} & \mbox{if } l=3k+2
\end{array}\right.$$
\end{cor}
\begin{proof}
It is clear that $I(K_n)\simeq\bigvee_{n-1}\mathbb{S}^0$. Now, $I(K_n\times K_2)$ is isomorphic to the simplicial complex obtained 
from two copies of the simplex $\Delta^{n-1}$ and adding an edge between the two copies of each vertex of $\Delta^{n-1}$. Therefore 
$I(K_n\times K_2)\simeq\bigvee_{n-1}\mathbb{S}^1$. The result follows from Theorem \ref{teomul}.
\end{proof}

\begin{cor}
$$I(\mu_l(K_n\times K_m))\simeq\left\lbrace
\begin{array}{cc}
\displaystyle\bigvee_{\frac{(n-1)^{k+1}(m-1)^{k+1}(nm-2)^{k}}{2^k}}\mathbb{S}^{4k+1}\vee\bigvee_{\left(\frac{(n-1)(m-1)(nm-2)}{2}\right)^k}\mathbb{S}^{4k} & \mbox{if } l=3k\\
\displaystyle\bigvee_{\frac{(n-1)^{k+1}(m-1)^{k+1}(nm-2)^{k}}{2^k}}\mathbb{S}^{4k+2} & \mbox{if } l=3k+1\\
\displaystyle\bigvee_{\left(\frac{(n-1)(m-1)(nm-2)}{2}\right)^{k+1}}\mathbb{S}^{4k+3} & \mbox{if } l=3k+2
\end{array}\right.$$
\end{cor}
\begin{proof}
The result follows from the fact that $I(K_2\times K_n\times K_m)$ has the homotopy type of a wedge of $(n-1)(m-1)(nm-2)/2$ 
spheres of dimension $3$ \citep[see Proposition 4.1.]{indcomplcatprod} and that $I(K_n\times K_m)$ has the homotopy type of a wedge of 
$(n-1)(m-1)$ circles \citep[see Proposition 3.4.]{homotopygoyal}.
\end{proof}

\begin{table}
\resizebox{\textwidth}{!}{
\begin{tabular}{|c|c|c|c|}\hline
\backslashbox[25mm]{$\;\;\;\;\;\;\;\;n$}{$l\;\;\;\;$}&$3k$&$3k+1$&$3k+2$\\
\hline
&&&\\
$6r$ &$[2r(2k+1)-1;2^{2k+1}],\;[4rk;2^{2k}]$&$[2r(2k+1);2^{2k+1}]$&$[4r(k+1)-1;2^{2k+2}]$\\
&&&\\
\hline
&&&\\
$6+1$ &$[2r(2k+1)+k-1;1],\;[k(4r+1);1]$&$[2r(2k+1)+k;1]$&$[(4r+1)(k+1)+k;1]$\\
&&&\\
\hline
&&&\\
$6r+2/6r+4$ &$[(2r+1)(2k+1)-1;1],\;[k(4r+2);1]$&$[(2r+1)(2k+1);1]$&$[(k+1)(4r+1)+k;1]$\\
&&&\\
\hline
&&&\\
$6r+3$ &$[(2r+1)(2k+1)-1;2^{k+1}],\;[(4r+1)k+k;2^{k}]$&$[(2r+1)(2k+1);2^{k+1}]$&$[(4r+1)(k+1)+k;2^{k+1}]$\\
&&&\\
\hline
&&&\\
$6r+5$&$[(2r+1)(2k+1)+k;1],\;[k(4r+3);1]$&$[(2r+1)(2k+1)+k+1;1]$&$[(4r+2)(k+1)+k;1]$\\
&&&\\
\hline
\end{tabular}}
\caption{Spheres in $I(\mu_l(C_n))$}\label{tablecn}
\end{table}

\begin{cor}
For any $l\geq1$ and $n\geq3$, $I(\mu_{l}(C_n))$ has the homotopy type of a wedge of spheres. The dimension and number of the spheres is given 
in Table \ref{tablecn} where $[a;b]$ means there are $b$ spheres of dimension $a$ in the wedge.
\end{cor}
\begin{proof}
The complex $I(C_n)$ has the homotopy type of a wedge of spheres \citep[see Proposition 4.6.]{kozlovdire} and  the sames 
goes for $I(C_n\times P_2)$ \citep[see Proposition 3.4.]{indcomplcatprod}, more precisely 
$$I(C_n)\simeq\left\lbrace\begin{array}{cc}
   \mathbb{S}^{r-1}\vee\mathbb{S}^{r-1}  & \mbox{ if } n=3r\\
    \mathbb{S}^{r-1}  & \mbox{ if } n=3r+1\\
    \mathbb{S}^r & \mbox{ if } n=3r+2
\end{array}\right.,\;\; I(C_n\times K_2)\simeq\left\lbrace 
\begin{array}{cc}
\displaystyle\bigvee_{4}\mathbb{S}^{4r-1} & \mbox{if } n=6r\\
\mathbb{S}^{4r} & \mbox{if } n=6r+1 \\
\mathbb{S}^{4r+1} & \mbox{if } n=6r+2\\
\mathbb{S}^{4r+1}\vee\mathbb{S}^{4r+1} & \mbox{if } n=6r+3 \\
\mathbb{S}^{4r+1} & \mbox{if } n=6r+4\\
\mathbb{S}^{4r+2} & \mbox{if } n=6r+5
\end{array}\right..$$
With this it is easy to obtain the dimensions and number of spheres given in Table \ref{tablecn}.
\end{proof}

\begin{prop}\label{propbipmu}
If $G$ is bipartite, then 
$$I(\mu_l(G))\simeq\left\lbrace\begin{array}{cc}
 I(G)^{*2k+1}\vee\Sigma I(G)^{*2k} & \mbox{ if } l=3k\\
 \Sigma I(G)^{*2k+1} & \mbox{ if } l=3k+1\\
 I(G)^{*2k+2} & l=3k+2
\end{array}
\right.$$
\end{prop}
\begin{proof}
The result follows from the fact that $G\times P_2\cong G\sqcup G$ and Theorem \ref{teomul}.
\end{proof}
\begin{cor}
For $n\in\{3r,3r-1\}$
$$I(\mu_l(P_n))\simeq\left\lbrace
\begin{array}{cc}
\mathbb{S}^{2kr+r-1}\vee\mathbb{S}^{kr+1} & \mbox{if } l=3k\\
\mathbb{S}^{2kr+r} & \mbox{if } l=3k+1\\
\mathbb{S}^{2(k+1)r-1} & \mbox{if } l=3k+2
\end{array}\right.,$$
for $n=3r+1$ $I(\mu_l(P_n))\simeq*$ for all $l$.
\end{cor}
\begin{proof}
It is known that $I(P_{3r+1})\simeq*$ and that $I(P_{3r+t})\simeq\mathbb{S}^{r-1}$ for $t=0,-1$ \citep[see Proposition 5.2.]{kozlovdire}. 
The rest follows from Proposition \ref{propbipmu}.
\end{proof}

\begin{cor}
If $T$ is a forest, then $I(\mu_l(T))$ is contractible or it has the homotoy type of a wedge of spheres.
\end{cor}
\begin{proof}
The result follows from Proposition \ref{propbipmu} and the fact that $I(T)$ is contractible or it has the homotopy type of a 
sphere \citep[see Corollary 6.1.]{Ehrenborg_2006}.
\end{proof}

\begin{cor}
If $G$ is the rectangular lattice with $n$ columns and at most $6$ rows, then $I(\mu_l(G))$ has the homotoy type of a wedge of spheres.
\end{cor}
\begin{proof}
The graph $G$ is bipartite, then the result follows from Proposition \ref{propbipmu} and the fact that $I(G)$ has 
the homotopy type of a wedge of spheres \citep{matsushitan4n5,matsushitan6}.
\end{proof}

\section{Independence complex of \texorpdfstring{$r$}{Î£}-iterated \texorpdfstring{$l$}{Î£}-Mycielskians}
In this section we give a homotopy formula for the independence complex of iterated Mycielskians, for this we first need to 
study the homotopy type of the independence complex of the Kronecker double cover of a Mycielskian.
\begin{figure}
\centering
\begin{tikzpicture}[line cap=round,line join=round,>=triangle 45,x=1.0cm,y=1.0cm]
\clip(1,3) rectangle (10,5);
\draw [rotate around={90:(3,4)}] (3,4) ellipse (0.71cm and 0.5cm);
\draw [rotate around={90:(8.5,4)}] (8.5,4) ellipse (0.71cm and 0.5cm);
\draw [rotate around={90:(4.5,4)}] (4.5,4) ellipse (0.71cm and 0.5cm);
\draw [rotate around={90:(7,4)}] (7,4) ellipse (0.71cm and 0.5cm);
\draw (3.35,3.5)-- (4.15,4.5);
\draw (4.15,3.5)-- (3.35,4.5);
\draw (8.15,4.5)-- (7.35,3.5);
\draw (8.15,3.5)-- (7.35,4.5);
\draw (2,4)-- (2.65,4.5);
\draw (2,4)-- (2.5,4);
\draw (2,4)-- (2.65,3.5);
\draw (9.5,4)-- (8.85,4.5);
\draw (9.5,4)-- (9,4);
\draw (9.5,4)-- (8.85,3.5);
\begin{scriptsize}
\fill [color=black] (2,4) circle (1.5pt);
\draw [color=black] (2,4.2) node {$x$};
\fill [color=black] (9.5,4) circle (1.5pt);
\draw [color=black] (9.5,4.2) node {$y$};
\draw [color=black] (3,4) node {$V_1$};
\draw [color=black] (4.5,4) node {$V_2$};
\draw [color=black] (7,4) node {$V_{2l+1}$};
\draw [color=black] (8.5,4) node {$V_{2l+2}$};
\draw [color=black] (5.75,4) node {$\cdots$};
\draw [color=black] (5.5,4) node {$\cdots$};
\draw [color=black] (5.25,4) node {$\cdots$};
\draw [color=black] (6,4) node {$\cdots$};
\draw [color=black] (6.25,4) node {$\cdots$};
\end{scriptsize}
\end{tikzpicture}
\caption{$H$}
\label{mulgtimesp2}
\end{figure}

\begin{theorem}\label{theogtimesp2}
$$I(\mu_l(G)\times P_2)\simeq\left\lbrace
\begin{array}{cc}
I(G\times P_2)^{*2k+1}\vee\Sigma^2I(G\times P_2)^{*2k} & \mbox{if } l=3k\\
\Sigma I(G\times P_2)^{*2k+1} & \mbox{if } l=3k+1\\
\Sigma^2 I(G\times P_2)^{*2k+2}    & \mbox{if } l=3k+2
\end{array}
\right.$$
\end{theorem}
\begin{proof}
We take $H$ the graph obtained from $G\times P_{2l+2}$ by adding two new vertices $x,y$ and adding the edges needed to get that 
$N_H(x)=V(G)\times\{1\}$ and $N_H(y)=V(G)\times\{2l+2\}$ (see Figure \ref{mulgtimesp2}). Then 
$\mu_l(G)\times P_2\cong H$ and $I(\mu_l(G)\times P_2)\cong I(H)$. 
We take $V_i=V(G)\times\{i\}$, then with similar arguments to those in the proof of Theorem \ref{teomul}, we have that 
$$I(H-x)\simeq I(H-x-V_3-\cdots-V_{3(2k)})$$
We take $W=H-x-V_3-\cdots-V_{3(2k)}$. 
Now, we have $3$ cases:
\begin{itemize}
    \item For $l=3k$, we have that $2l+2=3(2k)+2$. Thus $N_W((v,3(2k)+1))\subseteq N_W(y)$ for any $v$ in $V(G)$,
    and, by Lemma \ref{lemmvertdad}, we have that 
    $I(W)\simeq I(G\times P_2)^{*2k+1}$.
    \item For $l=3k+1$, we have that $2l+2=3(2k+1)+1$. Thus $N_W(y)=V_{3(2k+1)+1}$ and $I(W)\simeq\Sigma I(G\times P_2)^{*2k+1}$.
    \item For $l=3k+2$, we have that $2l+2=3(2k+2)$ and  $I(W)\simeq*$.
\end{itemize}
Taking $T=H-N_H[x]$, then:
\begin{itemize}
    \item For $l=3k$, we have that $I(T)\simeq\Sigma I(G\times P_2)^{*2k}$.
    \item For $l=3k+1$, we have that $I(T)\simeq*$.
    \item For $l=3k+2$, we have that $I(T)\simeq\Sigma I(G\times P_2)^{*2k+2}$.
\end{itemize}
For the case $l=3k$, notice that $I(G\times P_2)$ is connected, thus  
$$\mathrm{conn}\left(I(G\times P_2)^{*2k+1}\right)=\mathrm{conn}\left(H-x\right)>\mathrm{conn}\left(I(H-N_H[x])\right)=\mathrm{conn}\left(\Sigma I(G\times P_2)^{*2k}\right).$$ 
Then, regardless of the case, we have that
$$I(H-N_H[x])\longhookrightarrow I(H-x)$$ is null-homotopic and the result follows from Proposition \ref{inclunullhmt}.
\end{proof}

\begin{cor}\label{cormulrtimesp2}
$$I(\mu_l^r(G)\times P_2)\simeq\left\lbrace
\begin{array}{cc}
\Sigma^{f(k,r)}I(G\times P_2)^{*(2k+1)^r} & \mbox{if } l=3k+1\\
\Sigma^{2g(k,r)}I(G\times P_2)^{*(2k+2)^r}    & \mbox{if } l=3k+2
\end{array}
\right.$$
\end{cor}
\begin{proof}
The result follows from Theorem \ref{theogtimesp2} and Lemma \ref{lemfg}.
\end{proof}

\begin{theorem}
$$I(\mu_{3k+1}^r(G))\simeq\Sigma^{2(r-1)+\frac{f(k,r)-r}{2}}I(G)*I(G\times P_2)^{*kf(k,r)+1}$$
$$I(\mu_{3k+2}^r(G))\simeq\Sigma^{\frac{g(k,r+1)-1}{2}}I(G\times P_2)^{(k+1)(2k+2)^{r-1}}$$
\end{theorem}
\begin{proof}
The result follows from Lemma \ref{lemfg}, Theorem \ref{teomul} and Corollary \ref{cormulrtimesp2}.
\end{proof}

For the case $l=3k$ it does not seem feasible to give a closed formula. As example, by Theorems \ref{teomul} and 
\ref{theogtimesp2} we have that
$$I(\mu_{3k}^2(G))\simeq\left(I(G)*I(G\times P_2)^{*3k+1}\right)\vee\left(\Sigma^2I(G)*I(G\times P_2)^{*3k}\right)\vee\Sigma I(G\times P_2)^{*3k+1}\vee$$
$$\Sigma^3I(G\times P_2)^{*3k}\vee\bigvee_{i=0}^k\bigvee_{\binom{k}{i}}\Sigma^{2i+1}I(G\times P_2)^{*k^2+k-i}$$
While we do not have a close formula for $l=3k$ we can obtain the following Theorem.
\begin{theorem}
For $l=3k$, $I(\mu_{3k}^r(G))$ has the homotopy type of a wedge of spaces such that each space has the homotopy type of an iterated suspension 
of: a join of copies of $I(G)$; a join of copies of $I(G\times K_2)$; or or a join of both types of spaces.
\end{theorem}
\begin{proof}
This follows from Theorem \ref{teomul} and Theorem \ref{theogtimesp2}.
\end{proof}

Similar to the corollaries of the previous section, we have the following corollaries whose proofs are analogous to the ones in the 
previous section.

\begin{cor}
Let $G$ be any of the following graphs: $P_n,C_n,K_n,K_n\times K_m$. Then $I(\mu_l^r(G))$ has the homotopy type of a wedge of spheres.
\end{cor}

\begin{cor}
Let $T$ be a forest. Then $I(\mu_l^r(T))$ is contractible or it has the homotopy type of a wedge of spheres.
\end{cor}

\begin{cor}
If $G$ is the rectangular lattice with $n$ columns and at most $6$ rows, then $I(\mu_l^r(G))$ has the homotoy type of a wedge of spheres.
\end{cor}

\bibliographystyle{acm}
\bibliography{indcomplmckzn}
\end{document}